\documentclass[11pt]{amsart}

\usepackage{amsmath,amssymb,latexsym,soul,cite,mathrsfs}
\usepackage{tikz}
\usepackage{color,enumitem,graphicx}
\usepackage[colorlinks=true,urlcolor=blue,
citecolor=red,linkcolor=blue,linktocpage,pdfpagelabels,
bookmarksnumbered,bookmarksopen]{hyperref}
\usepackage[english]{babel}

\usepackage[left=3.5cm,right=3.5cm,top=3.8cm,bottom=3.8cm]{geometry}
\usepackage[hyperpageref]{backref}

\usepackage[colorinlistoftodos]{todonotes}
\makeatletter
\providecommand\@dotsep{5}
\def\listtodoname{List of Todos}
\def\listoftodos{\@starttoc{tdo}\listtodoname}
\makeatother

\usepackage[notref,notcite]{showkeys}

\numberwithin{equation}{section}

\newtheorem{theorem}{Theorem}[section]

\newtheorem{proposition}[theorem]{Proposition}
\newtheorem{lemma}[theorem]{Lemma}

\begin{document}

\title[Fractional Kirchhoff-type and method of sub-supersolutions]{Fractional Kirchhoff-type and method of sub-supersolutions}

\author{J. Vanterler da C. Sousa}

\address[J. Vanterler da C. Sousa]
{\newline\indent Aerospace Engineering, PPGEA-UEMA
\newline\indent
Department of Mathematics, DEMATI-UEMA
\newline\indent
São Luís, MA 65054, Brazil.}
\email{\href{vanterler@ime.unicamp.br,vanterlermatematico@hotmail.com}{vanterler@ime.unicamp.br,vanterlermatematico@hotmail.com}}

\pretolerance10000


\begin{abstract} In the present paper, we are interested in investigating the existence of positive solutions of a new class of fractional Kirchhoff via the sub and supersolutions technique. For this, we first need to investigate two results through lemmas.
\end{abstract}
\subjclass[2010]{35R11,35D05,35J60} 
\keywords{Fractional differential equation, Kirchhoff type problems, infinite semipositone problem, positive solution, sub and supersolutions.}

\maketitle
\section{Introduction and motivation}

This paper concerns with a new class of Kirchhoff-type fractional problems given by {\color{red}
\begin{equation}\label{eq1}
\left\{ 
\begin{array}{rcl}
\mathfrak{M}\left(\displaystyle\int_{\Lambda} \left\vert ^{\rm H}\mathfrak{D}_{0+}^{\alpha ,\beta ;\psi }\theta_{1}\right\vert ^{2} d\xi\right)\,\, ^{\rm H}\mathfrak{D}_{T}^{\alpha ,\beta ;\psi } \left(^{\rm H}\mathfrak{D}_{0+}^{\alpha ,\beta ;\psi }\theta_{1} \right)&=& \lambda \left(h(\theta_{1}) - \dfrac{1}{\theta_{1}^{\nu}} \right),\,\,in \,\,\Lambda=(0,T)\times(0,T)\\
\theta_{1} &=&0,\,\, on\,\, \xi\in\partial\Lambda
\end{array}%
\right.
\end{equation}}
where $\Lambda$ is a bounded interval in $\mathbb{R}^{2}$ with $C^{2+\gamma}$ boundary for some $0<\gamma<1$, $0<\nu<1$, $\mathfrak{M}:\mathbb{R}_{0}^{+}\rightarrow\mathbb{R}^{+}$ is a continuous and increasing function, $\lambda>0$ and $h:[0,\infty)\rightarrow \mathbb{R}$ is a continuous and nondecreasing function which is asymptotically $p$-linear at $\infty$. Furthermore, ${}^{\rm H}\mathfrak{D}^{\alpha,\beta;\psi}_{0+}(\cdot)$ and ${}^{\rm H}\mathfrak{D}^{\alpha,\beta;\psi}_{T}(\cdot)$ are the $\psi$-Hilfer fractional derivative (left-side and right-side) of order $\frac{1}{2}<\alpha\leq 1$ and type $0\leq \beta \leq 1$.

The theory of fractional differential equations involving $p$-Laplacian has gained prominence in recent years, particularly involving $\psi$-fractional operators \cite{Ezati,Ezati1,Nyamoradi,Nyamoradi1,Nori,Sousa4,Sousa80,Sousa3,Sousa,Sousa2,Sousa1}. The Kirchhoff-type problems have been investigated and discussed by several researchers, who can be consulted in the works
\cite{Goodrich,Ricceri,Sun12}. On the other hand, Kirchhoff's works via fractional operators have also gained prominence over the decades \cite{be1,be2,be3,be4}. Kirchhoff's model takes into account the changes in string length produced by transverse vibrations. The first classical studies devoted to Kirchhoff's equations were given by Bernstein
\cite{Bernstein} and Pohozaev \cite{Pohozaev}. See also the work of Lions \cite{Lions}, where he presents an abstract framework for the Kirchhoff model.

In 2005, Alves et al. \cite{Alves}, consider the existence of positive solutions for a quasilinear elliptic equation of Kirchhoff type given by
\begin{equation}\label{exemplo1}
    -M\left(\int_{\Omega} |\nabla |^{2} dx\right) \Delta u = f(x,u),\,\, in \,\,\Omega,\,\,u=0,\,\, on\,\partial\Omega.
\end{equation}
For more details on parameters and functions, see \cite{Alves}. The problem (\ref{exemplo1}) models different types of physical and biological systems. There are other important works of great relevance on Kirchhoff-type problems and can be obtained in \cite{Chung,Goodrich,novo12,novo11}.

Variational techniques to investigate the existence and multiplicity of solutions to Kirchhoff-type problems are indeed of great relevance. The sub and supersolution methods are another tool used to investigate the existence of positive solutions.

In 2014 Afrouzi et al. \cite{Afrouzi2} investigated the existence of positive solutions for the following Kirchhoff problems
\begin{equation}\label{eq111}
\left\{ 
\begin{array}{rcl}
-M\left(\displaystyle\int_{\Omega} |\nabla u|^{p} dx \right) div (|\nabla|^{p-2}\, \nabla u)&=&\lambda a(x)h(\theta_{1})-\mu,\,\, in\,\,\Omega  \\
u &=&0\,\, on\,\,x\in \partial\Omega.
\end{array}%
\right.
\end{equation}
using the sub and supersolution method. Other works on Kirchhoff-type problems using the sub and supersolution technique, see \cite{Zahmatkesh,Afrouzi,Yan,Han}.

Problem of type (\ref{eq111}) is related to the stationary version of the Kirchhoff equation
\begin{equation*}
    \rho \frac{\partial^{2} u}{ \partial t{2}}- \left(\frac{P_{0}}{h}+ \frac{E}{2L} \int_{0}^{L} |\partial u/ \partial x|^{2} dx \right) \frac{\partial^{2}u}{\partial x^{2}}=0
\end{equation*}
presented bt Kirchhoff in 1883 \cite{Kirchhoff}. This equation is an extension of the classical d'Alembert's wave equation by considering the effects of the changes in the lenght of the string during the vibrations.

Consider the following condition on the functions $\mathfrak{M}$ and $h$ given by:

($H_{1}$) $\zeta_{0}\leq \mathfrak{M}(t)\leq \zeta_{\infty}$,  $\forall t\in \mathbb{R}_{0}^{+}$, where $\mathbb{R}_{0}^{+}=[0,\infty)$,

($H_{2}$) $h\in C^{2}(0,\infty)$, $h(0) \geq 0$, $h'>0$, $\lim_{s\rightarrow\infty} \dfrac{h(s)}{s}=0$.

Motivated by the problems (\ref{exemplo1}) and (\ref{eq111}), and by questions open in theory, in this paper, we have as main result, to investigate the following result:

\begin{theorem}\label{principal} Consider $(H_{1})$ and $(H_{2})$ hold, and $h(0)<0$. Hence there exists positive constants $\mu_{1},\mu_{2}$ such that $\mu_{1}<\mu_{2}$ and the problem {\rm (\ref{eq1})} has no positive solution for $\lambda<\mu_{1}$ and has at least one positive solution for $\lambda>\mu_{2}$.
\end{theorem}

{\bf Remark:}
\begin{itemize}
    \item We can choose $\mathfrak{M}(t)=a+bt$ with $t=\displaystyle\int_{\Lambda} \left\vert^{\rm H}\mathfrak{D}^{\alpha,\beta ;\,\psi}_{0+}\theta_{1}\right\vert^{2}d \xi$, in the problem (\ref{eq1}), and also obtain the existence of positive solutions.

    \item A direct consequence of the result investigated here, in the limit $\alpha\rightarrow 1$, has an entire version of the problem. Furthermore, the freedom to choose $\psi$ allows us to discuss possible particular cases of the problem (\ref{eq1}).
\end{itemize}

In section 2, we present some definitions and essential results for the investigation of the main result. Finally, in section 3 we first investigate two {\bf Lemma \ref{exemplo1}} and {\bf Lemma \ref{eq111}}. In this sense, we conclude the article, attacking the existence of positive solutions via the technique of sub and supersolutions of problem (\ref{eq1}), basically, we are interested in the proof of the {\bf Theorem \ref{principal}}.

\section{ Mathematical background: auxiliary results}

Let $a=(a_{1},a_{2},a_{3})$, $b=(b_{1},b_{2},b_{3})$ and $\alpha=(\alpha_{1},\alpha_{2},\alpha_{3})$ where $0<\alpha_{1},\alpha_{2},\alpha_{3}<1$ with $a_{j}<b_{j}$, for all $j\in \left\{1,2,3 \right\}$. Also put $\Lambda=I_{1}\times I_{2}\times  I_{3}=[a_{1},b_{1}]\times [a_{2},b_{2}]\times [a_{3},b_{3}]$ where $b_{1},b_{2},b_{3}$ and $a_{1},a_{2},a_{3}$ positive constants. Consider also $\psi(\cdot)$ be an increasing and positive monotone function on $(a_{1},b_{1}),(a_{2},b_{2}),(a_{3},b_{3})$, having a continuous derivative $\psi'(\cdot)$ on $(a_{1},b_{1}],(a_{2},b_{2}],(a_{3},b_{3}]$. The $\psi$-Riemann-Liouville fractional partial integrals of $\phi\in \mathscr{L}^{1}(\Lambda)$ of order $\alpha$ $(0<\alpha<1)$ are given by \cite{J2,J1}:
\begin{itemize}
    \item 1-variable: right and left-sided
\begin{equation*}
    {\bf I}^{\alpha,\psi}_{a_{1}} \phi(\xi_{1})=\dfrac{1}{\Gamma(\alpha_{1})} \int_{a_{1}}^{\xi_{1}} \psi'(s_{1})(\psi(\xi_{1})- \psi(s_{1}))^{\alpha_{1}-1} \phi(s_{1}) ds_{1},\,\,to\,\,a_{1}<s_{1}<\xi_{1}
\end{equation*}
and
\begin{equation*}
    {\bf I}^{\alpha,\psi}_{b_{1}} \phi(\xi_{1})=\dfrac{1}{\Gamma(\alpha_{1})} \int_{\xi_{1}}^{b_{1}} \psi'(s_{1})(\psi(s_{1})- \psi(\xi_{1}))^{\alpha_{1}-1} \phi(s_{1}) ds_{1},\,\,to\,\,\xi_{1}<s_{1}<b_{1},
\end{equation*}
with $\xi_{1}\in[a_{1},b_{1}]$, respectively.
    
\item 3-variables: right and left-sided
\begin{eqnarray*}
    {\bf I}^{\alpha,\psi}_{a} \phi(\xi_{1},\xi_{2},\xi_{3})=\dfrac{1}{\Gamma(\alpha_{1})\Gamma(\alpha_{2})\Gamma(\alpha_{3})} \int_{a_{1}}^{\xi_{1}}
    \int_{a_{2}}^{\xi_{2}}
    \int_{a_{3}}^{\xi_{3}}
    \psi'(s_{1})\psi'(s_{2})\psi'(s_{3})
    (\psi(\xi_{1})- \psi(s_{1}))^{\alpha_{1}-1}\notag\\
    \times
    (\psi(\xi_{2})- \psi(s_{2}))^{\alpha_{2}-1}
    (\psi(\xi_{3})- \psi(s_{3}))^{\alpha_{3}-1}
    \phi(s_{1},s_{2},s_{3}) ds_{3}ds_{2}ds_{1},
\end{eqnarray*}
to $a_{1}<s_{1}<\xi_{1}, a_{2}<s_{2}<\xi_{2}, a_{3}<s_{3}<\xi_{3}$ and
\begin{eqnarray*}
    {\bf I}^{\alpha,\psi}_{b} \phi(\xi_{1},\xi_{2},\xi_{3})=\dfrac{1}{\Gamma(\alpha_{1})\Gamma(\alpha_{2})\Gamma(\alpha_{3})} \int_{\xi_{1}}^{b_{1}}
    \int_{\xi_{2}}^{b_{2}}
    \int_{\xi_{3}}^{b_{3}}
    \psi'(s_{1})\psi'(s_{2})\psi'(s_{3})
    (\psi(s_{1})-\psi(\xi_{1}))^{\alpha_{1}-1}\notag\\
    \times
    (\psi(s_{2})-\psi(\xi_{2}))^{\alpha_{2}-1}
    (\psi(s_{3})-\psi(\xi_{3}))^{\alpha_{3}-1}
    \phi(s_{1},s_{2},s_{3}) ds_{3}ds_{2}ds_{1},
\end{eqnarray*}
with $\xi_{1}<s_{1}<b_{1}, \xi_{2}<s_{2}<b_{2}, \xi_{3}<s_{3}<b_{3}$, $\xi_{1}\in[a_{1},b_{1}]$, $\xi_{2}\in[a_{2},b_{2}]$ and $\xi_{3}\in[a_{3},b_{3}]$, respectively.
\end{itemize}

On the other hand, let $\phi,\psi \in C^{n}(\Lambda)$ two functions such that $\psi$ is increasing and $\psi'(\xi_{j})\neq 0$ with $\xi_{j}\in[a_{j},b_{j}]$, $j\in \left\{1,2,3 \right\}$. The left and
right-sided $\psi$-Hilfer fractional partial derivative of $3$-variables of $\phi\in AC^{n}(\Lambda)$ of order $\alpha=(\alpha_{1},\alpha_{2},\alpha_{3})$ $(0<\alpha_{1},\alpha_{2},\alpha_{3}\leq 1)$ and type $\beta=(\beta_{1},\beta_{2},\beta_{3})$ where $0\leq\beta_{1},\beta_{2},\beta_{3}\leq 1$, are defined by \cite{J2,J1}
\begin{equation}\label{derivada1}
{^{\mathbf H}\mathfrak{D}}^{\alpha,\beta;\psi}_{a}\phi(\xi_{1},\xi_{2},\xi_{3})= {\bf I}^{\beta(1-\alpha),\psi}_{a} \Bigg(\frac{1}{\psi'(\xi_{1})\psi'(\xi_{2})\psi'(\xi_{3})} \Bigg(\frac{\partial^{3}} {\partial \xi_{1}\partial \xi_{2}\partial \xi_{3}}\Bigg) \Bigg) {\bf I}^{(1-\beta)(1-\alpha),\psi}_{a} \phi(\xi_{1},\xi_{2},\xi_{3})
\end{equation}
and
\begin{equation}\label{derivada2}
{^{\mathbf H}\mathfrak{D}}^{\alpha,\beta;\psi}_{b}\phi(\xi_{1},\xi_{2},\xi_{3})= {\bf I}^{\beta(1-\alpha),\psi}_{b} \Bigg(-\frac{1}{\psi'(\xi_{1})\psi'(\xi_{2})\psi'(\xi_{3})} \Bigg(\frac{\partial^{3}} {\partial \xi_{1}\partial \xi_{2}\partial \xi_{3}}\Bigg) \Bigg) {\bf I}^{(1-\beta)(1-\alpha),\psi}_{b} \phi(\xi_{1},\xi_{2},\xi_{3})
\end{equation}
where $a$ and $b$ are the same parameters presented in the definition of fractional integrals ${\bf I}_{b}^{\alpha;\psi}(\cdot)$ and ${\bf I}_ {a}^{\alpha;\psi}(\cdot)$.

Taking $a =0+$ in the definition of ${^{\mathbf H}\mathfrak{D}}^{\alpha,\beta;\psi}_{a}(\cdot)$, we have ${^{\mathbf H}\mathfrak{D}}^{\alpha,\beta;\psi}_{0+}(\cdot)$. During the paper we will use the following notation: ${^{\mathbf H}\mathfrak{D}}^{\alpha,\beta;\psi}_{a} \phi(\xi_{1},\xi_{2},\xi_{3}):= {^{\mathbf H}\mathfrak{D}}^{\alpha,\beta;\psi}_{a} \phi$, ${^{\mathbf H}\mathfrak{D}}^{\alpha,\beta;\psi}_{b} \phi(\xi_{1},\xi_{2},\xi_{3}):= {^{\mathbf H}\mathfrak{D}}^{\alpha,\beta;\psi}_{b} \phi$ and ${\bf I}_{a}^{\alpha;\psi}\phi(\xi_{1},\xi_{2},\xi_{3}):= {\bf I}_{a}^{\alpha;\psi}\phi$.

The $\psi$-fractional space is given by
$$
\mathbb{H}^{\alpha,\beta;\,\psi}_{p}(\Lambda)=\left\{\theta_{1}\in \mathscr{L}^{p}(\Lambda)~:~\left\vert^{\rm H}\mathfrak{D}^{\alpha,\beta;\,\psi}_{0+}\theta_{1}\right\vert\in \mathscr{L}^{p}(\Lambda)\right\}
$$ 
with the norm 
$$
||\theta_{1}||=||\theta_{1}||_{\mathbb{H}^{\alpha,\beta;\,\psi}_{p}(\Lambda)}=||\theta_{1}||_{\mathscr{L}^{p}(\Lambda)}+||^{\rm H}\mathfrak{D}^{\alpha,\beta;\,\psi}_{0+}\theta_{1}||_{\mathscr{L}^{p}(\Lambda)}.
$$ 
Denote by $\mathbb{H}^{\alpha,\beta;\,\psi}_{p,0}(\Lambda)$ the closure of $C_{0}^{\infty}(\Lambda)$ in $\mathbb{H}^{\alpha,\beta;\,\psi}_{p}(\Lambda).$
The $\psi$-fractional space $\mathbb{H}^{\alpha,\beta;\,\psi}_{p}(\Lambda)$ are separable and reflexive Banach spaces.

A function $\phi$ is said to be a sub-solution of problem (\ref{eq1}) if it is in $C^{2}(\Lambda)\cap C(\Lambda)$ such that $\phi=0$ on $\partial\Lambda$ and satisfies \cite{Han}
\begin{equation}\label{eq2.1}
    \mathfrak{M}\left(\displaystyle\int_{\Lambda} \left\vert ^{\rm H}\mathfrak{D}_{0+}^{\alpha ,\beta ;\psi }\phi\right\vert ^{2} d\xi\right)\,\,\int_{\Lambda}\,\, ^{\rm H}\mathfrak{D}_{0+}^{\alpha ,\beta ;\psi }\phi \,\,^{\rm H}\mathfrak{D}_{0+}^{\alpha ,\beta ;\psi }w\,\, d\xi \leq \int_{\Lambda} \lambda \left(h(\phi) - \dfrac{1}{\phi^{\nu}} \right) w\,\, d\xi,\,\,\forall w\in \mathcal{W}
\end{equation}
where $\mathcal{W}:=\left\{w\in C^{\infty}_{0}(\Lambda), w\geq 0\,\,in\,\,\Lambda \right\}$, and a function $\psi\in C^{2}(\Lambda)\cap C^{2}(\overline{\Lambda})$ is said to be a super-solution of problem (\ref{eq1}) if $\psi=0$ on $\partial\Lambda$ and satisfies \cite{Han}
\begin{equation}\label{eq2.2}
    \mathfrak{M}\left(\displaystyle\int_{\Lambda} \left\vert ^{\rm H}\mathfrak{D}_{0+}^{\alpha ,\beta ;\psi }\psi\right\vert ^{2} d\xi\right)\,\,\int_{\Lambda}\,\, ^{\rm H}\mathfrak{D}_{0+}^{\alpha ,\beta ;\psi }\psi \,\,^{\rm H}\mathfrak{D}_{0+}^{\alpha ,\beta ;\psi }w\,\, d\xi \geq \int_{\Lambda} \lambda \left(h(\psi) - \dfrac{1}{\psi^{\nu}} \right) w\,\, d\xi,\,\,\forall w\in \mathcal{W}.
\end{equation}

\begin{lemma}{\rm\cite{polo}}\label{lemma2.1} Assume that the $\mathfrak{M}:\mathbb{R}_{0}^{+}\rightarrow\mathbb{R}^{+}$ satisfies ${\rm(H_{1})}$. Furthermore, if $\theta_{1},\theta_{2}\in \mathbb{H}^{\alpha,\beta;\psi}_{p}(\Lambda)$ satisfies
\begin{eqnarray}\label{eq2.2}
    &&\mathfrak{M}\left(\displaystyle\int_{\Lambda} \left\vert ^{\rm H}\mathfrak{D}_{0+}^{\alpha ,\beta ;\psi }\theta_{1}\right\vert ^{p} d\xi\right)\,\,\int_{\Lambda}\,\, \left\vert^{\rm H}\mathfrak{D}_{T}^{\alpha ,\beta ;\psi }\theta_{1} \right\vert^{p-2}   \,\,^{\rm H}\mathfrak{D}_{0+}^{\alpha ,\beta ;\psi }\theta_{1} \,\,\,\,^{\rm H}\mathfrak{D}_{0+}^{\alpha ,\beta ;\psi }\varphi\,d\xi \notag\\&&\leq \mathfrak{M}\left(\displaystyle\int_{\Lambda} \left\vert ^{\rm H}\mathfrak{D}_{0+}^{\alpha ,\beta ;\psi }\theta_{2}\right\vert ^{p} d\xi\right)\,\,\int_{\Lambda}\,\, \left\vert ^{\rm H}\mathfrak{D}_{0+}^{\alpha ,\beta ;\psi }\theta_{2}\right\vert ^{p-2} \,\,^{\rm H}\mathfrak{D}_{0+}^{\alpha ,\beta ;\psi } \theta_{2}\,\,\,\,^{\rm H}\mathfrak{D}_{0+}^{\alpha ,\beta ;\psi }\varphi d\xi
\end{eqnarray}
for all $\varphi\in \mathbb{H}^{\alpha,\beta;\psi}_{p}(\Lambda)$, $\varphi\geq 0$, then $\theta_{1}\leq \theta_{2}$ in $\Lambda$.
\end{lemma}

\begin{proposition}{\rm\cite{polo,Afrouzi2}}\label{proposition2.2} Let $\mathfrak{M}:\mathbb{R}_{0}^{+}\rightarrow\mathbb{R}^{+}$ satisfying ${\rm(H_{1})}$. Assume that $h$ satisfies the sub-critical growth condition.
\begin{equation*}
    |h(x,t)|\leq c(1+|t|^{q-1}), \forall x\in \Lambda,\,\forall t\in \mathbb{R}
\end{equation*}
where $1<q<p^{*}$, and the function $h(x,t)$ is non-decreasing in $t\in \mathbb{R}$. If there exist a sub-solution $\underline{\theta_{1}}\in \mathbb{H}^{\alpha,\beta;\psi}_{p}(\Lambda)$ and a supersolution $\overline{\theta_{1}}\in\mathbb{H}^{\alpha,\beta;\psi}_{p}(\Lambda)$ of problem {\rm(\ref{eq1})}, then {\rm(\ref{eq1})} has a minimal solution ${\theta_{1}}_{*}$ and a maximal solution $\theta_{1}^{*}$ in the order interval $[{\theta_{1}}_{*},\theta_{1}^{*}]$, that is, $\underline{\theta_{1}}\leq {\theta_{1}}_{*}\leq \theta_{1}^{*} \leq \overline{\theta_{1}}$ and if $\theta_{1}$ is any solution of {\rm(\ref{eq1})} such that $\underline{\theta_{1}}\leq \theta_{1} \leq \overline{\theta_{1}}$, then ${\theta_{1}}_{*}\leq \theta_{1} \leq \theta_{1}^{*}$.
\end{proposition}

\begin{proposition} \label{proposition2.3} Let $\mathfrak{M}:\mathbb{R}_{0}^{+}\rightarrow\mathbb{R}^{+}$ satisfying ${\rm(H_{1})}$. Assume that $\underline{\theta_{1}}$ is a sub-solution and $\overline{\theta_{1}}$ is a supersolution or problem {\rm(\ref{eq1})} in the space $\mathbb{H}^{\alpha,\beta;\psi}_{p}(\Lambda)\cap L^{\infty}(\Lambda)$, and $\underline{\theta_{1}}\leq \overline{\theta_{1}}$ in $\Lambda$. If $h\in C(\overline{\Lambda}\times\mathbb{R},\mathbb{R})$ is nondecreasing in $t\in \left[ \inf_{\Lambda} \underline{\theta_{1}},\, \sup_{\Lambda} \overline{\theta_{1}}\right]$, then the conclusion of {\bf Proposition \ref{proposition2.2}} is valid.
\end{proposition}
\section{Existence of positive solution}

First, before proving two essential results in the discussion of the main result of this paper.

\begin{lemma}\label{lemma3.1} Let $\mathbb{R}^{+}\rightarrow\mathbb{R}$ be a continuous function and $\lim_{s\rightarrow\infty} \dfrac{h(s)}{s}=0$, then $\exists a,b>0$ such that
\begin{equation*}
    h(s)-\frac{1}{s^{\nu}}\leq as-b, \,\, \forall s\geq 0,
\end{equation*}
\end{lemma}
\begin{proof} For $c_{1}>0$ (fix), we consider
\begin{equation*}
    g(s)=\frac{h(s)}{s}-\frac{1}{s^{\nu+1}}-\frac{c_{1}}{s},\,\,\forall s>0.
\end{equation*}

{\color{red} As $\lim_{s\rightarrow\infty} g(s)=0$, so by definition we have $\left|\frac{h(s)}{s}-\frac{1}{s^{\nu+1 }}-\frac{c_{1}}{s}\right|\leq \varepsilon$.
Thus, $h(s)-\dfrac{1}{s^{\nu}}\leq \varepsilon s- c_{1}\,\, \forall s>A$. Let $s\in[0,A]$, as $h$ is continuous on $[0,A]$, then $h$ is bounded}
\begin{equation}\label{3.4}
    h(s)\leq c_{2}
\end{equation}
so
\begin{equation}\label{3.5}
    -\frac{1}{s^{\nu}}\leq - \frac{1}{A^{\nu}}.
\end{equation}
Using the inequalities (\ref{3.4}) and (\ref{3.5}), one has
\begin{equation*}
    h(s)-\frac{1}{s^{\nu}} \leq c_{2}- \frac{1}{A^{\nu}}.
\end{equation*}

Therefore, for all $s\in \mathbb{R}$, it follows that
\begin{equation*}
    h(s)-\frac{1}{s^{\nu}} \leq \varepsilon s- \left(c_{1}-c_{2}+\frac{1}{A^{\nu}}\right).
\end{equation*}

Taking $c_{1}$ large enough such that
\begin{equation*}
b= c_{1}-c_{2}+\frac{1}{A^{\nu}}>0.
\end{equation*}
Thus, we conclude that
\begin{equation*}
    h(s)-\frac{1}{s^{\nu}}< as-b.
\end{equation*}

\end{proof}

\begin{lemma}\label{lemma3.2} Let $f:\mathbb{R}^{+}\rightarrow\mathbb{R}$ be a continuous function such that $\lim_{s\rightarrow\infty} \dfrac{h(s)}{s}=0$, {\color{red} then for all $\lambda>0$, there exists $\zeta(\lambda)>0$ such that}
\begin{equation*}
\zeta_{0}\zeta(\lambda) \geq \lambda ( h(\zeta(\lambda) ||e||_{\infty})).
\end{equation*}
\end{lemma}

\begin{proof} Since $\lim_{s\rightarrow\infty} \dfrac{h(s)}{s}=0$, follows that $h(s)\leq \varepsilon\, s$. Taking $\varepsilon=\dfrac{\zeta_{0}}{ \lambda ||e||_{\infty}}>0$ and $s=\zeta(\lambda)||e||_{\infty}>A$. Hence, $\zeta(\lambda)\geq \dfrac{A}{||e||_{\infty}}$. Using $h(s)\leq \varepsilon s$, we concluded that
\begin{equation*}
    \lambda h(\zeta(\lambda) ||e||_{\infty})\leq \zeta(\lambda) \zeta_{0}.
\end{equation*}
\end{proof}

\begin{proof} {\rm \bf(Theorem \ref{principal})}. Since $\lim_{s\rightarrow\infty} \dfrac{h(s)}{s}=0$, $\exists a,b>0$ (constants) such that $h(s)-\dfrac{1}{s^{\nu}}<as-b$ (see {\bf Lemma \ref{lemma3.1}}). Let $\lambda_{1}$ (eigenvalue) and $\psi>0$ (eigenfunction) of operator $^{\rm H}\mathfrak{D}_{T}^{\alpha ,\beta ;\psi } \left(^{\rm H}\mathfrak{D}_{0+}^{\alpha ,\beta ;\psi }\theta_{1} \right)$ with Dirichlet boundary conditions. Consider a positive solution $\theta_{1}>0$ of (\ref{eq1}). Then
\begin{eqnarray}\label{eq2.1}
    &&\mathfrak{M}\left(\displaystyle\int_{\Lambda} \left\vert ^{\rm H}\mathfrak{D}_{0+}^{\alpha ,\beta ;\psi }\theta_{1}\right\vert ^{2} d\xi\right)\,\,\int_{\Lambda}\,\, ^{\rm H}\mathfrak{D}_{T}^{\alpha ,\beta ;\psi } \left(^{\rm H}\mathfrak{D}_{0+}^{\alpha ,\beta ;\psi }\theta_{1} \right) \psi\,\, d\xi \notag\\ &&\leq 
    \mathfrak{M}\left(\displaystyle\int_{\Lambda} \left\vert ^{\rm H}\mathfrak{D}_{0+}^{\alpha ,\beta ;\psi }\theta_{1}\right\vert ^{2} d\xi\right)\,\,\int_{\Lambda}\,\, \lambda \left(h(\theta_{1})-\dfrac{1}{\theta_{1}^{\nu}} \right) \psi\,\, d\xi \notag\\
    &&\leq 
    \mathfrak{M}\left(\displaystyle\int_{\Lambda} \left\vert ^{\rm H}\mathfrak{D}_{0+}^{\alpha ,\beta ;\psi }\theta_{1}\right\vert ^{2} d\xi\right)\,\,\lambda\int_{\Lambda}\,\, \left(a \theta_{1}-b\right) \psi\,\, d\xi \notag\\
    &&\leq 
    \lambda \zeta_{\infty}\,\,\int_{\Lambda}\,\, \left(a \theta_{1}-b\right) \psi\,\, d\xi.
\end{eqnarray}

But, not that
\begin{equation*}
    \int_{\Lambda}\,\, ^{\rm H}\mathfrak{D}_{T}^{\alpha ,\beta ;\psi } \left(^{\rm H}\mathfrak{D}_{0+}^{\alpha ,\beta ;\psi }\theta_{1} \right) \psi\,\, d\xi = \int_{\Lambda}\,\, \theta_{1} \left(^{\rm H}\mathfrak{D}_{T}^{\alpha ,\beta ;\psi } \left(^{\rm H}\mathfrak{D}_{0+}^{\alpha ,\beta ;\psi }\psi \right)\right) d\xi = \int_{\Lambda} \lambda_{1} \theta_{1} \psi\,\, d\xi.
\end{equation*}

Thus,{\color{red}
\begin{eqnarray}\label{eq2.1}
\mathfrak{M}\left(\displaystyle\int_{\Lambda} \left\vert ^{\rm H}\mathfrak{D}_{0+}^{\alpha ,\beta ;\psi }\theta_{1}\right\vert ^{2} d\xi\right)\,\,\int_{\Lambda}\,\, \left(\lambda_{1}-\lambda \zeta_{\infty} a\right)\theta_{1} \psi\,\, d\xi\leq 
    \zeta_{\infty}\int_{\Lambda}\,\, \left(-\lambda b\right) \psi\,\, d\xi.
    \end{eqnarray}}

Note that, this is impossible if $\lambda<\dfrac{\lambda_{1}}{\zeta_{\infty} a}$. Therefore, we conclude the first part of the theorem. Now, let $\phi>0$ and $||\phi||_{\infty}=1$ and $\phi=\lambda^{r} \psi^{2/1+\nu}$, where the parameter $r\in \left(\dfrac{1}{1+\nu},1 \right)$. Then, since $^{\rm H}\mathfrak{D}_{0+}^{\alpha ,\beta ;\psi } \phi= \lambda^{r}\, \dfrac{\Gamma(3+\nu/1+\nu)}{\Gamma(3+\nu-\alpha-\alpha\nu/1+\nu)}\,\psi^{\frac{2}{1+\nu}-\alpha}$ (see \cite{J1}) one has
\begin{eqnarray*}
    &&\mathfrak{M}\left(\displaystyle\int_{\Lambda} \left\vert ^{\rm H}\mathfrak{D}_{0+}^{\alpha ,\beta ;\psi }\phi\right\vert ^{2} d\xi\right)\,\,\int_{\Lambda}\,\, ^{\rm H}\mathfrak{D}_{0+}^{\alpha ,\beta ;\psi }\phi \,\,\,^{\rm H}\mathfrak{D}_{0+}^{\alpha ,\beta ;\psi }w\,\, d\xi \notag\\ &&=  \frac{\Gamma(3+\nu/1+\nu)}{\Gamma(3+\nu-\alpha-\alpha\nu/1+\nu)} \mathfrak{M}\left(\displaystyle \int_{\Lambda} \left\vert ^{\rm H}\mathfrak{D}_{0+}^{\alpha ,\beta ;\psi }\phi\right\vert ^{2} d\xi\right)\,\,\int_{\Lambda}\,\, \psi^{\frac{2}{1+\nu}-\alpha}\, \,^{\rm H}\mathfrak{D}_{0+}^{\alpha ,\beta ;\psi }w\,\, d\xi \notag\\ && \leq
    (\zeta_{\infty}-\zeta_{0}) \lambda^{r} \frac{\Gamma(3+\nu/1+\nu)}{\Gamma(3+\nu-\alpha-\alpha\nu/1+\nu)} \int_{\Lambda}\,\, \psi^{\frac{2}{1+\nu}-\alpha}\, \,|^{\rm H}\mathfrak{D}_{0+}^{\alpha ,\beta ;\psi }w|^{2} d\xi.
\end{eqnarray*}

Thus $\psi$ is a sub-solution of problem (\ref{eq1}) if
\begin{eqnarray}
    &&(\zeta_{\infty}-\zeta_{0}) \lambda^{r} \frac{\Gamma(3+\nu/1+\nu)}{\Gamma(3+\nu-\alpha-\alpha\nu/1+\nu)} \int_{\Lambda}\,\, \psi^{\frac{2}{1+\nu}-\alpha}\, \,|^{\rm H}\mathfrak{D}_{0+}^{\alpha ,\beta ;\psi }w|^{2} d\xi \notag \\ &\leq &\int_{\Lambda} \lambda \left(h(\phi)-\frac{1}{\phi^{\nu}} \right)w\,\, d\xi.
\end{eqnarray}

Let $\delta>0$, $\zeta>0$ and $\mu>0$ be such that $|^{\rm H}\mathfrak{D}_{0+}^{\alpha ,\beta ;\psi }w|^{2}\geq \zeta$ in $\overline{\Lambda}_{\delta}$ and $\psi^{2/1+\nu}\in [\mu,1]$ in $\Lambda/\overline{\Lambda}_{\delta}$ where
\begin{equation*}
    \overline{\Lambda}_{\delta}:=\left\{\xi\in\Lambda: d(\xi,\partial\Lambda) \leq \delta \right\},
\end{equation*}
since $|^{\rm H}\mathfrak{D}_{0+}^{\alpha ,\beta ;\psi }\zeta|\neq 0$ on $\partial\Lambda$. Hence, in $\overline{\Lambda}_{\delta}$, if $\lambda\geq 1$, yields
\begin{equation*}
    -\zeta_{0} \lambda^{r}\, \frac{\Gamma(3+\nu/1+\nu)}{\Gamma(3+\nu-\alpha-\alpha\nu/1+\nu)} \psi^{\frac{2}{1+\nu}-\alpha} \,\, |^{\rm H}\mathfrak{D}_{0+}^{\alpha ,\beta ;\psi }w|^{2}\leq \lambda w \left(-\frac{1}{(\lambda^{r} \psi^{\frac{2}{1+\nu}-\alpha})^{\nu}}  \right),
\end{equation*} 
since $1-r-r\nu<0$. Also, in $\overline{\Lambda}_{\delta}$, if $\lambda \geq 1$, one has
\begin{equation*}
    \zeta_{\infty } \lambda^{r}\, \frac{\Gamma(3+\nu/1+\nu)}{\Gamma(3+\nu-\alpha-\alpha\nu/1+\nu)} \psi^{\frac{2}{1+\nu}-\alpha} \,\, |^{\rm H}\mathfrak{D}_{0+}^{\alpha ,\beta ;\psi }w|^{2}\leq \lambda w h\left(\lambda^{r} \psi^{\frac{2}{1+\nu}-\alpha}  \right).
\end{equation*}

Hence, in $\overline{\Lambda}_{\delta}$ follows that
\begin{eqnarray*}
    &&\mathfrak{M}\left(\displaystyle\int_{\Lambda} \left\vert ^{\rm H}\mathfrak{D}_{0+}^{\alpha ,\beta ;\psi }\phi\right\vert ^{2} d\xi\right)\,\,\int_{\Lambda}\,\, ^{\rm H}\mathfrak{D}_{0+}^{\alpha ,\beta ;\psi }\phi \,\,\,^{\rm H}\mathfrak{D}_{0+}^{\alpha ,\beta ;\psi }w\,\, d\xi \notag\\ &&\leq\int_{\Lambda} \lambda w \left[h\left(\lambda \psi^{\frac{2}{1+\nu}-\alpha}\right)-\frac{1}{\left(\lambda^{r} \psi^{\frac{2}{1+\nu}-\alpha} \right)^{\nu}}\right]d\xi\notag\\
    &&= \int_{\Lambda} \lambda \left(h(\phi)-\frac{1}{\phi^{\nu}} \right) w\,\, d\xi.
\end{eqnarray*}

Next in $\Lambda/\overline{\Lambda}_{\delta}$, since $\psi^{\frac{2}{1+\nu}}\geq \mu$, one has
\begin{eqnarray}
\lambda \left(h(\phi)-\frac{1}{\phi^{\nu}} \right)\geq \lambda \left(h(\lambda^{r} \mu)- \frac{1}{(\lambda^{r} \mu)^{\nu}} \right).
\end{eqnarray}

But if $\lambda \geq 1$
\begin{eqnarray}\label{eq3.2}
    &&\mathfrak{M}\left(\displaystyle\int_{\Lambda} \left\vert ^{\rm H}\mathfrak{D}_{0+}^{\alpha ,\beta ;\psi }\phi\right\vert ^{2} d\xi\right)\,\,\int_{\Lambda}\,\, ^{\rm H}\mathfrak{D}_{0+}^{\alpha ,\beta ;\psi }\phi \,\,\,^{\rm H}\mathfrak{D}_{0+}^{\alpha ,\beta ;\psi }w\,\, d\xi 
    \notag\\ && \leq
    (\zeta_{\infty}-\zeta_{0}) \lambda^{r} \frac{\Gamma(3+\nu/1+\nu)}{\Gamma(3+\nu-\alpha-\alpha\nu/1+\nu)} \int_{\Lambda}\,\, \psi^{\frac{2}{1+\nu}-\alpha}\, \,|^{\rm H}\mathfrak{D}_{0+}^{\alpha ,\beta ;\psi }w|^{2} d\xi
    \notag\\ &&\leq \int_{\Lambda} \lambda \left(h(\lambda^{r} \mu)- \frac{1}{(\lambda^{r}\mu)^{\nu}} \right) w\,\, d\xi,
\end{eqnarray}
since $r<1$. Hence, if $\lambda \geq 1$, in $\Lambda/\overline{\Lambda}_{\delta}$, yields
\begin{eqnarray}\label{eq3.3}
    &&\mathfrak{M}\left(\displaystyle\int_{\Lambda} \left\vert ^{\rm H}\mathfrak{D}_{0+}^{\alpha ,\beta ;\psi }\phi\right\vert ^{2} d\xi\right)\,\,\int_{\Lambda}\,\, ^{\rm H}\mathfrak{D}_{0+}^{\alpha ,\beta ;\psi }\phi \,\,\,^{\rm H}\mathfrak{D}_{0+}^{\alpha ,\beta ;\psi }w\,\, d\xi 
    \notag\\&&\leq \int_{\Lambda} \lambda \left(h(\phi)- \frac{1}{(\phi)^{\nu}} \right) w\,\, d\xi.
\end{eqnarray}

Using the inequalities (\ref{eq3.2})-(\ref{eq3.3}), its follows that $\phi$ is a positive sub-solution of (\ref{eq1}).

The next step is to obtain a supersolution $\Xi \geq \phi$.

Since, $\lim_{s\rightarrow\infty} \dfrac{h(s)}{s}=0$, then for all $\lambda>0$, there exists $\zeta(\lambda)>0$ such that $\zeta(\lambda) \geq \lambda h(\zeta(\lambda) ||e||_{\infty})$ (see {\bf Lemma \ref{lemma3.2}}), where $e\in C^{1}(\overline{\Lambda})$ is the unique positive solution of the boundary value problem
\begin{equation*}
    ^{\rm H}\mathfrak{D}_{T}^{\alpha ,\beta ;\psi } \left( ^{\rm H}\mathfrak{D}_{0+}^{\alpha ,\beta ;\psi } e\right)=1,\,\, in\,\, \Lambda=(0,T)\times(0,T)
\end{equation*}
with $e=0$. To discuss our result, it is known that $e>0$ in $\Lambda$ and $\dfrac{\partial e}{\partial \eta}<0$ on $\partial\Lambda$. Let $\Xi:= \zeta(\lambda)e$. Then, follows that
\begin{eqnarray}
    \mathfrak{M}\left(\displaystyle\int_{\Lambda} \left\vert ^{\rm H}\mathfrak{D}_{0+}^{\alpha ,\beta ;\psi }\psi\right\vert ^{2} d\xi\right)\,\,\int_{\Lambda}\,\, ^{\rm H}\mathfrak{D}_{0+}^{\alpha ,\beta ;\psi }\psi \,\,\,^{\rm H}\mathfrak{D}_{0+}^{\alpha ,\beta ;\psi }w\,\, d\xi   &=& \zeta(\lambda) \mathfrak{M}\left(\displaystyle\int_{\Lambda} \left\vert ^{\rm H}\mathfrak{D}_{0+}^{\alpha ,\beta ;\psi }\psi\right\vert ^{2} d\xi\right)\,\, \int_{\Lambda} w\,\, d\xi \notag\\
    &\geq& \int_{\Lambda} (\lambda h(\zeta(\lambda))||e||_{\infty} )w\,\, d\xi
    \notag\\ &\geq&  \zeta(\lambda) \zeta_{0} \int_{\Lambda} w\,\, d\xi
    \notag\\ &\geq& \int_{\Lambda} \lambda \left(h(\Xi)-\frac{1}{\Xi^{\nu}} \right),\,\,\forall w\in \mathcal{W}.
\end{eqnarray}

So, $\Xi$ is a supersolution. So, $\Xi=\zeta(\lambda)\, e \geq \phi$ in $\overline{\Lambda}$, for $\zeta(\lambda)$ large. Therefore, we conclude that for $\lambda\geq 1$, problem (\ref{eq1}) has a positive solution $\theta_{1}\in [\phi, \Xi]$.
\end{proof}

\section*{Acknowledgements}

{\color{red} The author is very grateful to the anonymous reviewers for their useful comments that led to improvement of the manuscript.}

\section*{Data Availability Statement}

Data sharing not applicable to this paper as no data sets were generated or analyzed during the current study.

\section*{Conflicts of Interest} 
The author declare that they have no known competing financial interests or personal relationships that could have appeared to influence the work reported in this paper.




\begin{thebibliography}{}
\bibitem{Ezati} Ezati, R., and N. Nyamoradi. Existence of solutions to a Kirchhoff $\psi$-Hilfer fractional $p$-Laplacian equations. Math. Meth. Appl. Sci. 44.17 (2021): 12909-12920.


\bibitem{Ezati1} Ezati, R., and N. Nyamoradi. Existence and multiplicity of solutions to a $\psi$-Hilfer fractional $p$-Laplacian equations. Asian-European J. Math. (2022): 2350045.


\bibitem{novo11} {\color{red} Choudhuri, D. Existence and Hölder regularity of infinitely many solutions to a $p$-Kirchhoff-type problem involving a singular nonlinearity without the Ambrosetti–Rabinowitz (AR) condition. Z. Angew. Math. Phys. 72 (2021): 1-26. }

\bibitem{novo12} {\color{red} Zuo, J., D. Choudhuri, and D. D. Repovs. On critical variable-order Kirchhoff type problems with variable singular exponent. J. Math. Anal. Appl. 514.1 (2022): 126264.}

\bibitem{Nyamoradi} Nyamoradi, N. and V. Ambrosio. Existence and non-existence results for fractional Kirchhoff Laplacian problems. Anal. Math. Phys. 11.3 (2021): 1-25.

\bibitem{Ricceri} Ricceri, B. On an elliptic Kirchhoff-type problem depending on two parameters. J. Glob. Optim. 46(4), 543–549 (2010)

\bibitem{Sun12} Sun, J.J., Tang, C.L.: Existence and multiplicity of solutions for Kirchhoff type equations. Nonlinear Anal. 74, 1212–1222 (2011).

\bibitem{Nyamoradi1} Nyamoradi, N. and M. Kirane. Existence of solutions of fractional $p$-Laplacian systems with different critical Sobolev‐Hardy exponents. Math. Meth. Appl. Sci. 43.17 (2020): 10237-10248.

\bibitem{Nori} Nori, A. A., N. Nyamoradi, and N. Eghbali. Multiplicity of Solutions for Kirchhoff Fractional Differential Equations Involving the Liouville-Weyl Fractional Derivatives. J. Contem. Math. Anal. (Armenian Academy of Sciences) 55.1 (2020): 13-31.


\bibitem{Sousa4} Sousa, J. Vanterler da C., M. Aurora P. Pulido, and E. Capelas de Oliveira. Existence and Regularity of Weak Solutions for $\psi$-Hilfer Fractional Boundary Value Problem. Mediter. J. Math. 18.4 (2021), 1-15.

\bibitem{Sousa80} Sousa, J. Vanterler da C. Existence and uniqueness of solutions for the fractional differential equations with $p$-Laplacian in $\mathbb{H}^{\nu,\eta;\psi}_{p}$. J. Appli. Anal. Comput. 12(2) (2022), 622-661. 

\bibitem{Sousa3} Sousa, J. Vanterler da C., Leandro S. Tavares, and César E. Torres Ledesma. A variational approach for a problem involving a $\psi$-Hilfer fractional operator. J. Appl. Anal. Comput. 11.3 (2021), 1610-1630.


\bibitem{Sousa} Sousa, J. Vanterler da C., C. T. Ledesma, M. Pigossi, Jiabin Zuo. Nehari Manifold for Weighted Singular Fractional $p$-Laplace Equations. Bull. Braz. Math. Soc. (2022): 1-31.

\bibitem{Sousa2} Sousa, J. Vanterler da C. Nehari manifold and bifurcation for a $\psi$-Hilfer fractional $p$‐Laplacian. Math. Meth. Appl. Sci. (2021). doi.org/10.1002/mma.7296.

\bibitem{Sousa1} Sousa, J. Vanterler da C., Jiabin Zuo, and Donal O'Regan. The Nehari manifold for a $\psi$-Hilfer fractional $p$-Laplacian. Applicable Anal. (2021), 1-31.

\bibitem{Alves} Alves, C. O., F. J. S. A. Corrêa, and To Fu Ma. Positive solutions for a quasilinear elliptic equation of Kirchhoff type. Comput. Math. Appl. 49.1 (2005): 85-93.

\bibitem{Chung} Chung, N. T. An existence result for a class of Kirchhoff type systems via sub and supersolutions method. Appl. Math. Lett. 35 (2014): 95-101.


\bibitem{be1} Appolloni, L., G. Molica Bisci, and S. Secchi. On critical Kirchhoff problems driven by the fractional Laplacian. Calc. Var. Partial Diff. Equ. 60.6 (2021): 1-23.

\bibitem{be2} Fiscella, A. A fractional Kirchhoff problem involving a singular term and a critical nonlinearity. Adv. Nonlinear Anal. 8.1 (2019): 645-660.

\bibitem{be3} Xiang, M., B. Zhang, and X. Zhang. A critical Kirchhoff type problem involving the fractional $p$-Laplacian in $\mathbb{R}^{N}$. Complex Var. Ellip. Equ. 63.5 (2018): 652-670.


\bibitem{be4} Tyagi, J. Eigenvalue problem for fractional Kirchhoff Laplacian. Rendiconti Lincei 29.1 (2018): 195-203.



\bibitem{Goodrich} Goodrich, C. S. A one-dimensional Kirchhoff equation with generalized convolution coefficients. J. Fixed Point Theory Appl. 23.4 (2021): 1-23.

\bibitem{Afrouzi2} Afrouzi, G. A., N. T. Chung, and S. Shakeri. Positive solutions for a semipositone problem involving nonlocal operator. Rendiconti del Seminario Matematico della Università di Padova 132 (2014): 25-32.


\bibitem{Zahmatkesh} Zahmatkesh, H., S. Shakeri, and A. Hadjian. An existence result for a class of nonlocal infinite semipositone problem. Boletín de la Sociedad Matemática Mexicana 27.3 (2021): 1-7.


\bibitem{Afrouzi} Afrouzi, G. A., Chung N. T., and S. Shakeri. Existence and non-existence results for nonlocal elliptic systems via sub-supersolution method. Funkcialaj Ekvacioj 59.3 (2016): 303-313.


\bibitem{Yan} Yan, B., Donal O'Regan, and Ravi P. Agarwal. The existence of positive solutions for Kirchhoff-type problems via the sub-supersolution method. Analele Universitatii" Ovidius" Constanta-Seria Matematica 26.1 (2018): 5-41.

\bibitem{Han} Han, X, Dai, G : On the sub-supersolution method for $p(x)$-Kirchhoff type equations. J. Inequ. Appl . 2012, 283(2012).

\bibitem{polo} Sousa, J. Vanterler da C. Sousa, Kishor D. Kucche and Juan J. Nieto. Existence and multiplicity of solutions for fractional $\kappa(\xi)$-Kirchhoff-type equation. (preprint) 2022.

\bibitem{Bernstein} Bernstein, S. Sur une classe d'equations fonctionnelles aux derivees partielles, (in Russian with French summary), Bull. Acad. Sci. URSS, Set. Math. 4, 17-26, (1940).

\bibitem{Pohozaev} Pohozaev, S. On a class of quasilinear hyperbolic equations, Math. Sborniek 96, 152-166, (1975).


\bibitem{Lions} Lions, J. L. On some questions in boundary value problems of mathematical physics, In Proceedings of International Symposium on Continuum Mechanics and Partial Differential Equations, Rio de Janeiro 1977, Math. Stud. (Edited by de la Penha and Medeiros), pp. 284-346, vol. 30, North-Holland, (1978).

\bibitem{Kirchhoff} Kirchhoff, G.: Mechanik. Teubner, Leipzig (1883).



\bibitem{J2} Sousa, J. Vanterler da C., and E. Capelas de Oliveira. On the stability of a hyperbolic fractional partial differential equation. Diff. Equ. Dyn. Sys. (2019): 1-22.


\bibitem{J1} Sousa, J. Vanterler da C., and E. Capelas de Oliveira. On the $\psi$-Hilfer fractional derivative. Commun. Nonlinear Sci. Numer. Simul. 60 (2018), 72-91. 



\end{thebibliography}
\end{document}